\documentclass[11pt]{article}
\def\RR{\mathbb{R}}

\def\uu{\text{\bf u}}
\def\ww{\text{\bf w}}
\def\ff{\text{\bf f}}
\def\AA{\text{\bf A}}

\usepackage[latin1]{inputenc}
\usepackage[USenglish,english,american]{babel}
\usepackage{amsfonts,setspace}
\usepackage{framed}

\usepackage{cancel}
\usepackage{amssymb}
\usepackage{anysize} 
\usepackage{multicol}
\usepackage[dvipsnames]{xcolor}
\usepackage{graphicx}
\usepackage{float}
\usepackage{bbm}
\usepackage{amsmath}
\usepackage{mathrsfs}
\usepackage{mathtools}
\usepackage{amsthm}
\usepackage{hyperref}
\usepackage{geometry}
\hypersetup{
    colorlinks,
    citecolor=blue,
    filecolor=magenta,
    linkcolor=red,
    urlcolor=black
}

\theoremstyle{theorem}
\newtheorem{theorem}{Theorem}
\theoremstyle{corollary}
\theoremstyle{proposition}
\newtheorem{proposition}{Proposition}
\theoremstyle{lemma}
\theoremstyle{definition}
\theoremstyle{Assumption}
\theoremstyle{remark}
\newtheorem{remark}{Remark}
\theoremstyle{obs}

\begin{document}

\begin{center}
\begin{Large}{\bf Reconstruction for multi-wave imaging in attenuating media with large damping coefficient}
\end{Large}\\
\vspace{1em}
\begin{large}{\bf Benjam\'in Palacios}\\
{\small Department of Mathematics, University of Washington, Seattle, WA, USA}
\end{large}

\end{center}
\renewcommand{\abstractname}{} 
\begin{abstract} 
\noindent{\sc Abstract.} In this article we study the reconstruction problem in TAT/PAT on an attenuating media. Namely, we prove a reconstruction procedure of the initial condition for the damped wave equation via Neumann series that works for arbitrary large smooth attenuation coefficients extending the result of Homan in \cite{Homan}. We also illustrate the theoretical result by including some numerical experiments at the end of the paper.
\end{abstract}

\noindent{\it Keywords:} Multiwave imaging; Neumann series; damped wave equation; geometric optics\\


\centerline{\sc 1. Introduction}
{\it Thermoacoustic} and {\it Photoacoustic tomography} are coupled physics medical imaging techniques that consist in the application of a harmless electromagnetic pulse to some target tissue which causes a slightly increment in the local temperature, making the tissue to expand and produce pressure waves which are then recorded and used to reconstruct optical parameters of the region of interest. The goal of combining two different types of waves is to take advantage of the good contrast of the electromagnetic absorption parameters and the good resolution the ultrasound data provides, to finally obtain a tomogram endowed with both desired features. It's then natural to split the problem into two steps, the first one being the inversion of the ultrasound measurements to get some internal information related with the absorption of the electromagnetic radiation, and the second step, also call {\it quantitative TAT/PAT}, consisting in the inversion of such internal data to finally reconstruct the absorption coefficient which constitutes the image of the inside of the body. This paper concerns only in the first step.

The classical setting of this problem is to assume that waves propagate through the free space and there have been many results that follow this assumption \cite{2007Finch},\cite{2009Hristova},\cite{2009Kunyansky},\cite{TAT},\cite{TATbrain},\cite{Homan}. The validity of such hypothesis depends on the devices and setting used to acquire the ultrasound data, requiring then in this case to use detectors whose perturbation to the waves can be neglected. This is not always the case though \cite{2007BdryCondPAT}. Such fact has motivated the current study of the case where the acoustic waves propagate inside an enclosure, this is when one allows the waves to interact with some boundary or interface, as can be seen in \cite{2015StefanovYang}, \cite{2015AcostaMontalto}, \cite{Nguyen2015}, where the interactions are modeled by restricting the domain where waves propagate and by considering boundary conditions (e.g. Neumann, Robin). We believe that the analysis carried out in this article might be applied to the enclosure case to obtain reconstruction formulas when an interior damping coefficient is taken into account. 

It's a fact that the attenuation of the ultrasound waves affects the quality of the reconstruction, see for example \cite{2006Patch}, and
recently there has been a growing interest in trying to incorporate and compensate for the effect of the damping in TAT/PAT \cite{2016AcostaMontalto},\cite{Homan},\cite{2012PATatt},\cite{2011Roitner}. 
It's also known that the attenuation in biological tissues is strongly frequency-dependent, and modeling such relationship is not a trivial task. Many models have been proposed to represent such damping effect, most of them giving way to fractional derivatives and consequently to integro-differential operators (see \cite{2011fractionalwave} and \cite{2011Ammari} for a review of such models). We instead consider here a more simple non-frequency dependent model, namely the damped wave equation, motivated by the work done in \cite{Homan}.

We introduce a new Neumann series reconstruction formula that allows us to recover the source term in the attenuated TAT/PAT problem for the damped wave equation allowing more general attenuation coefficients. This problem was first addressed in \cite{Homan} where the author used a sharp time reversal introduced originally by Stefanov and Uhlmann in \cite{TAT}, and proved that under the assumption that the attenuation coefficient is smooth and sufficiently small, the error of the back projection
is a contraction.

Let $\Omega\subset\RR^n$ be a bounded domain with smooth boundary and let $\Lambda_a:f\mapsto u|_{[0,T]\times\partial\Omega}$ denote the measurement operator where $u$ satisfies
\begin{equation}\label{TAT_eq}
\left\{\begin{array}{rcl}
(\partial^2_t +a\partial_t- c^2\Delta)u &=& 0\; \text{ in }(0,T)\times\RR^n,\\
u|_{t=0}&=&f,\\
\partial_tu|_{t=0}&=&-af.
\end{array}\right.
\end{equation}
The particular form of the initial condition comes from the equivalence of system \eqref{TAT_eq} with the problem of finding $f$ in
$$(\partial^2_t +a\partial_t- c^2\Delta)u = f(x)\delta'(t),\; \text{ in }\RR\times\RR^n,\quad u = 0\;\text{ for }t<0,$$
as it's explained in \cite{Homan}. In such article, the time reversal operator associated to the damped wave equation is given by $\hat{A}_a:h\mapsto v(0,\cdot)$, where $v$ is the solution of
\begin{equation}\label{time_reversal}
\left\{\begin{array}{rcl}
(\partial^2_t +a\partial_t- c^2\Delta)v &=& 0\; \text{ in }(0,T)\times\Omega,\\
v|_{t=T}&=&\phi,\\
\partial_tv|_{t=T}&=&0,\\
v|_{(0,T)\times\partial\Omega}&=&h,\end{array}\right.
\end{equation}
with $\phi$ the harmonic extension of 
$h|_{t=T}$ to $\Omega$. 
Following this approach it is possible to get a reconstruction formula via Neumann series for small enough attenuations (see \cite[Theorem 2.3]{Homan}). The proof is based on the continuity of the error operator $\hat{K}_a = \text{Id}-\hat{A}_a\Lambda_a$, in terms of the attenuation coefficient, which provided that $(\Omega,c^{-2}dx^2)$ is non-trapping, $T$ larger than the supremum of the length of interior geodesics for the metric $c^{-2}dx^2$,
and $\|a\|_\infty$ is sufficiently small, then $\|\hat{K}_a\|<1$. 

The novelty of this article is that we get a Neumann series reconstruction that allows arbitrary large bounded smooth attenuations as long as they satisfy that their support is contained in $\overline{\Omega}$. We use a different time reversal technique than in \cite{Homan}, namely we backwardly solve the damped wave equation with a damping coefficient given by $-a(x)$ (see \eqref{damped_time_reversal}). When such IBVP is solved from $t=T$ to $t=0$, the sign of the damping term flips producing an attenuation of the reversed wave. Such fact allows us to better control the energy in the iterative process of the Neumann series. To prove the error operator is a contraction we employ the microlocal approach developed in \cite{TATbrain}, this is, we get microlocal energy estimates of the high frequency part for the transmitted waves that are used to show that after some time most of the initial energy lies outside the domain. The difficulty here is to treat the new term that the attenuation coefficient makes appear in the energy. Another consequence of the damping is that as in Homan's article, we also need the measurement time $T$ to be greater than the supremum of the length of the interior geodesics, which matches the time needed to get stability for the damped TAT/PAT problem, and which doubles the time needed in the undamped case. 
\bigskip

\centerline{\sc 2. Preliminaries}

Given a domain $U$ and a scalar function $u(t,x)$, we define the local energy of $\uu = [u,u_t]$ at time $t$ as
$$E_U(\uu(t)) = \int_{U}(|\nabla_x u|^2 + c^{-2}|u_t|^2)dx.$$
We also define the extended energy functional to be
$$\mathcal{E}_U(\uu,\tau) =E_U(\uu(\tau)) + 2\int_{[0,\tau]\times U}ac^{-2}|u_t(t)|^2dtdx.$$
If $\uu$ is a solution of the damped wave equation in the whole space, it is well known that for any $U\subset \RR^n $ such that $u$ vanishes for all time on its boundary, the former energy functional is non-increasing due to the attenuation coefficient $a\geq0$, while the latter is conserved. Indeed, the first statement follows from the usual computation of the energy where we multiply equation \eqref{TAT_eq} by $2c^{-2}u_t$ and integrate over $U$, so then, integrating by parts we arrive to the energy inequality
$$\frac{d}{dt}E_U(\uu(t)) = -2\int_{U}ac^{-2}|u_t(t)|^2dx\leq 0.$$
If in addition we integrate in the interval $(0,\tau)$, for any $\tau>0$, we get the extended energy conservation. 
\\

The energy space $\mathcal{H}(U)$ of initial conditions is defined to be the completion of $C^\infty_0(U)\times C^\infty_0(U)$ under the energy norm
$$\|\ff\|^2_{\mathcal{H}(U)} = \int_{U}(|\nabla_x f_1|^2 + c^{-2}|f_2|^2)dx.$$
with $\ff = [f_1,f_2]$. 
Notice that $\mathcal{H}(U) = H_D(U)\oplus L^2(U)$, where we write $L^2(U)=L^2(U;c^{-2}dx)$ the $L^2$ space for the measure $c^{-2}dx$.

Recall that denoting $\uu = [u,u_t]$, it is possible to write \eqref{TAT_eq} as the following system
$$\uu_t = \text{\bf P}_a\uu,\quad \text{\bf P}_a = \left(\begin{matrix} 0&I\\c^2\Delta& -a\end{matrix}\right),$$
where $\text{\bf P}_a$ defines a strongly continuous semigroup, and given initial conditions $\ff = [f_1,f_2]\in \mathcal{H}(\Omega)$ the solution of the previous system takes the form $\uu = e^{t\text{\bf P}_a}\ff$.
\bigskip

\centerline{\sc 3. Main result}

Let $\Omega\subset \RR^n$ be a strictly convex bounded domain with smooth boundary and let's consider a solution $u$ of the initial value problem
\begin{equation}\label{general IBVP}
\left\{\begin{array}{rcl}
(\partial^2_t +a\partial_t- c^2\Delta)u &=& 0\; \text{ in }(0,T)\times\RR^n,\\
u|_{t=0}&=&f_1,\\
u_t|_{t=0}&=&f_2,
\end{array}\right.
\end{equation}
where we assume that $c,a\in C^\infty(\RR^n)$ are such that $c>0, a\geq 0$ and
$c -1 = a = 0 \text{ in }\RR^n\setminus\Omega.$ In the case of TAT/PAT we take $(f_1,f_2) = (f,-af)$ for some $f\in H_{D}(\Omega)$.

We denote the measurement operator 
$$\Lambda_a:\mathcal{H}(\Omega) \to H^1((0,T)\times\partial\Omega),\quad \Lambda_a\ff := u|_{(0,T)\times\partial\Omega}.$$
By letting $h= \Lambda_a\ff$, we consider the time reversed wave $v(x,t)$ to be the solution of the backward problem
\begin{equation}\label{damped_time_reversal}
\left\{\begin{array}{rcl}
(\partial^2_t - a\partial_t- c^2\Delta)v &=& 0\; \text{ in }(0,T)\times\Omega,\\
v|_{t=T}&=&P_{\partial\Omega}h(T),\\
v_t|_{t=T}&=&0,\\
v|_{(0,T)\times\partial\Omega}&=&h,\end{array}\right.
\end{equation}
where $P_{\partial\Omega}$ is the harmonic extension operator in $\Omega$. Notice the sign in the attenuation term. We will see that by solving \eqref{damped_time_reversal} back in time the solution is also attenuated, and consequently we can control the energy during the time reversal and the subsequent iterations of the Neumann series. We set 
$$\AA_a:H^{1}_{(0)}([0,T]\times\partial\Omega)\to \mathcal{H}(\Omega) \cong H^1_0(\Omega)\oplus L^2(\Omega)$$
$$\AA_ah=[v(0,\cdot),v_t(0,\cdot)] =:[A_a^1h,A_a^2h]$$ the new time reversal operator which is a continuous map by \cite{Lasiecka} and finite speed of propagation (The subscript in $H^1_{(0)}$ stands for functions vanishing at $t=0$).

Notice that due to uniqueness of solutions, the error function $w$ satisfying
\begin{equation}\label{error_function}
\left\{\begin{array}{rcl}
(\partial^2_t - c^2\Delta)w &=& -a(u_t+v_t)\; \text{ in }(0,T)\times\Omega,\\
w|_{t=T}&=&u^T - \phi,\\
\partial_tw|_{t=T}&=&u_t^T,\\
w|_{(0,T)\times\partial\Omega}&=&0,\end{array}\right.
\end{equation}
with $(u^T,u_t^T) = (u(T,\cdot),u_t(T,\cdot))$, is such that $u = v+w$ in $(0,T)\times\Omega$. 
We then define the error operator for the TAT/PAT problem to be
$${\bf K}_a:\mathcal{H}(\Omega)\to \mathcal{H}(\Omega),\quad {\bf K}_af = [w(0,\cdot),w_t(0,\cdot)],$$ 
this is, ${\bf A}_a\Lambda_a = \text{Id} - {\bf K}_a$.

In contrast with the non-attenuated case we can no longer say the error operator is a composition of a compact and a unitary operator as in \cite{TAT}, since now $w$ satisfies a wave equation with nontrivial right hand side which depends on $u$. To overcome this fact we use the microlocal approach given in \cite{TATbrain} which is based on energy estimates of the high frequency part of the solution. The proof of the reconstruction formula is then reduced to show that at time $T$ a significant part of the initial energy has left the domain.

It's well known we can only detect singularities of the initial conditions that propagate through geodesics that hit the boundary at times less than the measurement time $T$. In our case though,
we consider a more restricted condition for the visibility of singularities due to the effect of the attenuation coefficient, which requires both ends of those geodesics to hit $\partial\Omega$ in a non-tangent way. When such coefficient is non-zero it is not possible to orthogonally project the initial data into the initial conditions that generate $\gamma_{(x,\xi)}$ and $\gamma_{(x,-\xi)}$ respectively (see \cite[section 4.2]{TATbrain}), therefore we cannot split the analysis and work with each of those branches independently. This condition on the visibility of singularities can be compared with the fact that the damped wave equation can not be extended to negative times, making necessary to double the measurement time needed in the undamped case to get uniqueness (see \cite[Theorem 3.1]{Homan}).

We set
$$T_0(\Omega) = \sup\{|\gamma|_{g}: \gamma \subset \bar{\Omega} \text{ geodesic for the metric } g = c^{-2}dx^2\}.$$
Requiring that $T_0(\Omega)$ is finite (this is ($\Omega,c^{-2}dx^2$) non-trapping),
the visibility of all singularities is guaranteed provided we take enough time to record the boundary data. Consequently, assuming the measurement time $T>T_0(\Omega)$ means we are able to detect both parts of every singularity produced by the initial condition, or in other words, both signals originated by any singularity are ``visible'' in time less than $T$. 

The main result is the following
\begin{theorem}\label{NS_free_space}
Assume \text{\rm ($\Omega$, $c^{-2}dx^2)$} is strictly convex, and $T_0(\Omega)<T<\infty$.
Then ${\bf K}_a$ is a contraction in $\mathcal{H}(\Omega)$ and we get the following reconstruction formula for the thermoacoustic problem
$$\ff = \sum^\infty_{m=0}{\bf K}_a^m{\bf A}_ah\quad h:=\Lambda_a\ff.$$
\end{theorem}
\bigskip
\centerline{\sc 4. Geometric Optics}

\noindent{\bf 4.1 Parametrix for the Cauchy problem.}
The proof of Theorem \ref{NS_free_space} is based on energy estimates for the high frequency part of the solution where the analysis is done by constructing approximate solutions localized near null bicharacteristics. Homan in \cite{Homan} got a parametrix solution for \eqref{general IBVP} in the case of $\ff = (f,-af)$, following an standard argument. We need though a parametrix for the general case $\ff = (f_1,f_2)$. The construction below is based on \cite[\S 4.1]{TATbrain} where the only difference resides in the transport equations \eqref{transport}.

We are looking for a solution of the form
$$u(t,x) = (2\pi)^{-n}\sum_{\sigma=\pm}\int e^{i\phi^\sigma(t,x,\eta)}\big(A^\sigma_1(t,x,\eta)\hat{f}_1(\eta) + |\eta|^{-1}A^\sigma_2(t,x,\eta)\hat{f}_2(\eta)\big)d\eta,$$
where $\hat{f}_i$ stands for the Fourier transform of $f_i$.
To find the equations that $\phi^{\pm}$ and $A^\pm_{j}$, j=1,2, must satisfy we first compute:
$$\square_a u = (2\pi)^{-n}\sum_{\sigma=\pm}\int e^{i\phi^\sigma}\big([I^\sigma_{1,0} + I^\sigma_{1,1} + I^\sigma_{1,2}]\hat{f}_1 + |\eta|^{-1}[I^\sigma_{2,0} + I^\sigma_{2,1} + I^\sigma_{2,2}]\hat{f}_2\big)d\eta$$
where
\begin{align*}
I^\sigma_{j,2} &= - A_j^\sigma((\partial_t\phi^\sigma)^2 - c^2|\nabla_y\phi^\sigma|^2),\\
I^\sigma_{j,1} &= 2i[(\partial_t\phi^\sigma)(\partial_tA^\sigma_j) - c^2\nabla_y\phi^\sigma\cdot\nabla_yA^\sigma_j] + iA^\sigma_j\square_a\phi^\sigma,\\
I^\sigma_{j,0} &= \square_aA^\sigma_j.
\end{align*}
Considering classical amplitude functions given by the asymptotic expansions
$$A^\sigma_j(t,x,\eta)\sim \sum_{k\geq 0}A^\sigma_{j,k}(t,x,\eta),\quad \sigma=\pm,$$
with $A^\sigma_{j,k}$ homogeneous of degree $-k$ in $\eta$, we would like to find those $A^\sigma_{j,k}$ so that $u$ solves \eqref{general IBVP} up to a smooth terms. We then choose $\phi^\sigma$ and $A^\sigma_{j,k}$ so that the terms of same order of homogeneity in $\eta$ cancel each other. The phase functions must satisfy the eikonal equation which takes into account the second order terms $I^\sigma_{j,2}$, and we endow it with initial conditions
\begin{align}\label{eikonal}
\left\{\begin{matrix} \mp\partial_t\phi^\pm &=& c|\nabla_x\phi^\pm|\\ \phi^\pm|_{t=0} &=& x\cdot\eta.\end{matrix}\right.
\end{align}
Consequently, we obtain $I^\sigma_{j,2} = 0$.

To get rid of the next terms with less order of homogeneity we have to solve transport equations. We define the vector field
\begin{equation}\label{transport_op}
X^\sigma = 2(\partial_t\phi^\sigma)\partial_t - 2c^2\nabla_x\phi^\sigma\cdot\nabla_x.
\end{equation}
Then, the coefficients of both amplitude functions must satisfy 
\begin{equation}\label{transport}
X^\sigma A^\sigma_{j,0} + A^\sigma_{j,0}\square_a\phi^\sigma = 0,\quad\text{and}\quad X^\sigma A^\sigma_{j,k} + A^\sigma_{j,k}\square_a\phi^\sigma = - \square_aA^\sigma_{j,k-1} \text{ for }k\geq 1.
\end{equation}
Since we must have that $u|_{t=0} = f_1$, this is
$$f_1(x) = (2\pi)^{-n}\int e^{ix\cdot\eta}\big((A^+_1 + A^-_1)|_{t=0}\hat{f}_1(\eta) + |\eta|^{-1}(A^+_2 + A^-_2)|_{t=0}\hat{f}_2(\eta) \big)d\eta,$$
we need the condition 
$$A^+_1 + A^-_1 = 1,\quad A^+_2 + A^-_2=0,\quad \text{at }t=0.$$Analogously, from $u_t|_{t=0} = f_2$ and since $\phi^\sigma$ satisfies \eqref{eikonal}, we have
\begin{align*}
f_2(x) &= (2\pi)^{-n}\int e^{ix\cdot\eta}\big([ic|\eta|(-A^+_1 + A^-_1) + \partial_t(A^+_1 + A^-_1)]|_{t=0}\hat{f}_1(\eta)\\
&\hspace{7em}+[ic(-A^+_2 + A^-_2) +|\eta|^{-1}\partial_t(A^+_2 + A^-_2)]|_{t=0}\hat{f}_2(\eta) \big)d\eta,
\end{align*}
thus, at $t=0$ we require
$$ic|\eta|(-A^+_1 + A^-_1) + \partial_t(A^+_1 + A^-_1) = 0,\quad ic(-A^+_2 + A^-_2) +|\eta|^{-1}\partial_t(A^+_2 + A^-_2) = 1.$$
We then consider initial conditions given by the following system at $t=0$ which can be solved iteratively:
\begin{equation}\label{IC_amplitud1}
\left\{\begin{array}{rcl} A^+_{1,0} + A^-_{1,0} &=& 1\\  A^+_{1,k} + A^-_{1,k} &=&0,\; k\geq 1  \end{array}\right.\hspace{.5em}\left\{\begin{array}{rcl} A^+_{1,0} - A^-_{1,0} &=& 0\\  A^+_{1,k} - A^-_{1,k} &=& ic^{-1}|\eta|^{-1}\partial_t(A^+_{1,k-1} + A^-_{1,k-1}),\;k\geq1,\end{array}\right.
\end{equation}
\begin{equation}\label{IC_amplitud2}
\left\{\begin{array}{rcl} A^+_{2,0} + A^-_{2,0} &=& 0\\  A^+_{2,k} + A^-_{2,k} &=&0,\; k\geq 1  \end{array}\right.\hspace{.5em} \left\{\begin{array}{rcl} A^+_{2,0} - A^-_{2,0} &=& i/c\\  A^+_{2,k} - A^-_{2,k} &=& ic^{-1}|\eta|^{-1}\partial_t(A^+_{2,k-1} + A^-_{2,k-1}),\;k\geq1.\end{array}\right.
\end{equation}
We solve the transport equations in \eqref{transport} on integral curves of $X^\sigma$ as longs as the eikonal equation \eqref{eikonal} is solvable, and imposing the initial conditions from \eqref{IC_amplitud1} and \eqref{IC_amplitud2}. In particular, at $t=0$, the leading term are given by $A^+_{1,0} = A^-_{1,0} = \frac{1}{2}$, and $A^+_{2,0} = -A^-_{2,0} = \frac{i}{2c}$.

We can iteratively use the previous construction by solving the eikonal equation in small increments on time and get parametrices $u_+, u_-$ defined on $[0,T]$. In addition, by assuming the wave front set of $\ff$ lies inside a small conical neighborhood of some $(x_0,\xi_0)\in S^*\bar{\Omega}$, their supports can be assumed to be contained in small neighborhoods of the respective branches of the geodesics issued from $(x_0,\xi_0)$, and their wave front sets contained in small neighborhoods of the bicharacteristics issued from $(0,x_0,1,\xi_0)$ and $(0,x_0,1,-\xi_0)$ respectively. 

If we restrict the previous solution to the boundary we obtain an approximate representation (up to smooth term) of the measurement operator $\Lambda_a$ as a sum of Fourier Integral Operators, this is $\Lambda_a\ff \cong \Lambda_a^+\ff + \Lambda_a^-\ff$, where
\begin{align*}
[\Lambda_a^{\pm}\ff](t,x) &= [\Lambda^\pm_{a,1}f_1 + \Lambda^\pm_{a,2}f_2](t,x)\\
&=(2\pi)^{-n}\int e^{i\phi^\pm(t,x,\eta)}\big(A^\pm_1(t,x,\eta)\hat{f}_1(\eta) + |\eta|^{-1}A^\pm_2(t,x,\eta)\hat{f}_2\big)d\eta\Big|_{\partial\Omega}.
\end{align*}
It follows from the same arguments as in \cite{Homan} and \cite{TAT} that their canonical relations are of graph type. In consequence, since $\Lambda^\pm_{a,1}$ are FIOs of order 0 and $\Lambda^\pm_{a,2}$ are of order -1, writing $h = \Lambda_a\ff$, we get the estimate
\begin{equation}\label{FIO_reg}
\|h\|_{H^{0}([0,T]\times\partial\Omega)}\leq C\|\ff\|_{H^{0}(\Omega)\times H^{-1}(\Omega)}.
\end{equation}
Moreover, following the same arguments as in \cite[Theorem 3.2]{Homan}, we have the stability estimate:
\begin{proposition}[Stability]\label{stability} Assume $T_0(\Omega)<T<\infty$. Then
\begin{equation}\label{Homan_stability}
\|\ff\|_{\mathcal{H}(\Omega)}\leq C\|\Lambda_a\ff\|_{H^1([0,T]\times\partial\Omega)},\quad\forall \ff\in \mathcal{H}(\Omega).
\end{equation}
\end{proposition}

\begin{remark}\label{remark0}
From the assumption on the wave front set of $\ff$, we can assume that $\Lambda^+_a\ff$ and $\Lambda^-_a\ff$ have disjoint wave front set which is a consequence of the fact that bicharacteristics don't self intersect.
\end{remark}

\bigskip
\noindent{\bf 4.2 Parametrix at the boundary. }
We now want to have a pseudodifferential representation of the Dirichlet-to-Neumann map. Let's pick any $(x_0,\xi_0)\in S^*\bar{\Omega}$ and denote by 
$(t_1,x_1,1,\xi_1)\in T^*(\RR\times\Omega)$ the point in phase space where the bicharacteristic issued from $(0,x_0,1,\xi_0)$ hits the boundary. We also denote by $(t_1,x_1,1,(\xi_1)')$ its projection onto $T^*(\RR\times\partial\Omega)$. Close to the boundary and in a neighborhood of $x_1$ we choose local coordinates $x=(x',x^n)$ such that $\partial\Omega$ is given by $x^n=0$ and $x^n>0$ in $\RR^n\setminus\Omega$. The following analysis can also be done for the other bicharacteristic issued from $(0,x_0,1,-\xi_0)$.

Consider a compactly supported distribution $h$ on $\RR\times\partial\Omega$ with $WF(h)$ contained in a small conic neighborhood of $(t_1,x_1,1,(\xi^1)')$. As in \cite{TATbrain} we can get an approximate solution of \eqref{TAT_eq} outside $\Omega$ for the transmitted wave with positive wave speed $c(x)|\xi|$, as an FIO applied to $h$. Since $\partial \Omega$ is an invisible boundary for the forward wave, such transmitted wave would be the same as the solution $u_+$ defined above. The transmitted wave related to the positive wave speed has the form
\begin{equation}\label{bdary_parametrix}
u^+_T = (2\pi)^{-n}\int e^{i\varphi_+(t,x,\tau,\xi')}b_+(t,x,\tau,\xi')\hat{h}_+(\tau,\xi')d\tau d\xi',
\end{equation}
where $\hat{h}_+ = \int_{\RR\times\RR^{n-1}}e^{-i(- t\tau + x'\cdot\xi')}h(t,x')dtdx'$. In particular, the phase functions $\varphi_+$ must satisfy the eikonal equation plus boundary conditions on $x^n=0$:
\begin{equation}\label{eikonal+bc for +}
\partial_t\varphi_+ + c(x)|\nabla_{x}\varphi_+| = 0,\quad \varphi_+|_{x^n=0} = - t\tau + x'\cdot\xi',
\end{equation}
where the choice of sign in the eikonal equation is such in order to agree with $\phi^+$ in \eqref{eikonal}. 
In the case of the negative sound speed, the transmitted wave $u^-_T$ has the same form \eqref{bdary_parametrix} but  interchanging $\hat{h}_+$ with $\hat{h}_-= \int_{\RR\times\RR^{n-1}}e^{-i(+ t\tau + x'\cdot\xi')}h(t,x')dtdx'$ and $\varphi_+$ with $\varphi_-$ which satisfies
\begin{equation}\label{eikonal+bc for -}
-\partial_t\varphi_- + c(x)|\nabla_{x}\varphi_-| = 0,\quad \varphi_-|_{x^n=0} =  t\tau + x'\cdot\xi'.
\end{equation}
As in the previous construction we consider the amplitude functions to be classical, this is $b_{\pm} \sim \sum_{k\geq 0}b^{\pm}_k$ with $b^{\pm}_k$ homogeneous of degree $-k$ in $\tau$ and $\xi'$. It then follows that $b_{\pm}$ satisfy analogous equations to \eqref{transport} with boundary conditions 
$$b^{\pm}_0 = 1,\quad b^{\pm}_k=0,\; k\geq 1,\quad \text{at } x^n = 0.$$

We use the previous to locally define the Dirichlet-to-Neumann (DN) maps for the positive and negative wave speed as the following $\Psi$DOs of order 1 (this is due to \eqref{eikonal+bc for +})
$$N_{\pm}:h\mapsto \frac{\partial u^{\pm}_T}{\partial x^n}\Big|_{\RR\times\partial\Omega}$$ 
This definitions are local since they depend on the choice of local boundary coordinates.
From \eqref{bdary_parametrix} we get their principal symbols are the same and given by
$$\sigma(N_\pm) = \text{i}\frac{\partial \varphi_\pm}{\partial x^n}\Big|_{\RR\times\partial\Omega}  =\text{i}\sqrt{c^{-2}\tau^2 - |\xi'|^2},$$
thus they are elliptic in the hyperbolic conic set $c^{-1}|\tau|>|\xi'|$.
Notice that in contrast with \cite{TATbrain}, since the boundary doesn't perturb the propagation of the wave there is no distinction between the incoming and outgoing DN maps.

\begin{remark}\label{remark1}
The previous construction agrees up to smooth terms with the approximate solutions for the Cauchy problem $u_+$ and $u_-$, in a neighborhood of the boundary, when we take $h = \Lambda^{\pm}_a\ff$ respectively.
\end{remark}
\bigskip

\centerline{\sc 5. Proof of Theorem \ref{NS_free_space}}

By multiplying \eqref{error_function} by $w_t = (u_t - v_t)$ and integrating on $(0,T)\times\Omega$ we get that the local energy of $w$ satisfies 
\begin{align*}E_\Omega(\ww(0)) &= E_\Omega(\ww(T)) + 2\int_{[0,T]\times\Omega} ac^{-2}(u_t+v_t)(u_t - v_t)dtdx\\
&= E_\Omega(\ww(T)) + 2\int_{[0,T]\times\Omega}  ac^{-2}|u_t|^2dtdx - 2\int_{[0,T]\times\Omega} ac^{-2}|v_t|^2dtdx\\
&\leq \|[u^T-\phi,u_t^T]\|^2_{\mathcal{H}(\Omega)}+ 2\int_0^T\int_\Omega ac^{-2}|u_t|^2dtdx.
\end{align*}
Moreover, since $\phi$ is harmonic and $u^T|_{\partial\Omega} = \phi|_{\partial\Omega}$ we have
$$(u^T - \phi,\phi)_{H_D(\Omega)} = 0,$$
which implies 
$$\|[u^T - \phi,u_t^T]\|^2_{\mathcal{H}(\Omega)} = E_\Omega(\uu(T))-\|\phi\|²_{H_D(\Omega)}.$$
Then
\begin{align}\label{K_ineq}
E_\Omega(\ww(0))\leq \mathcal{E}_\Omega(\uu,T).
\end{align}

We now state the main step in the proof of the theorem which says that after time $T>T_0(\Omega)$ a considerable part of the energy is outside the domain. In the enclosure case, this is when boundary conditions are imposed, a similar proposition would have to be proven by considering an estimate relating the initial energy and the energy absorbed on the boundary.
\begin{proposition} Let $\uu$ be a solution of \eqref{general IBVP} with initial condition $\ff\in\mathcal{H}(\Omega)$. There exists $C>1$ so that
$$\|\ff\|^2_{\mathcal{H}(\Omega)} \leq CE_{\Omega^c}(\uu(T)).$$
\end{proposition}

\begin{proof}
Recall that for any bounded domain $U$ with smooth boundary and for $t'\leq t\leq t''$ with $t'<t''$, if $\uu$ is a solution of the damped wave equation then
\begin{equation}\label{energy1}
\mathcal{E}_U(\uu,t'') = \mathcal{E}_U(\uu,t') + 2\mathfrak{R}\int_{[t',t'']\times\partial U}u_t\frac{\partial \overline{u}}{\partial\nu}dtdS,
\end{equation}
where $\nu$ is the outward normal unit-vector to $\partial U$. 

Let's assume for a moment that $WF(\ff)$ is contained in a conical neighborhood of some $(x_0,\xi_0)\in S^*\bar{\Omega}$ and as before we denote by $(t^\pm_1,x^\pm_1)$ the times and points where the respective branches of the geodesic issued from $(x_0,\xi_0)$ make contact with the boundary. We now want to estimate the energy transmitted outside $\Omega$ up to compact operator applied to $\ff$, and at time $t=T$. For a large ball $B$, the energy of the solution $\uu = e^{t{\bf P}_a}\ff$ of \eqref{TAT_eq} in $U = B\setminus\Omega$ is given by
\begin{equation}\label{energy2}
E_{\Omega^c}(\uu,t_2) = 2\mathfrak{R}\int_{[0,t_2]\times\partial \Omega}\frac{\partial u}{\partial t}\frac{\partial \overline{u}}{\partial\nu}dtdS.
\end{equation}
where there is no term at time $t=0$ since $\uu(0)$ vanishes outside $\Omega$. Here $\nu$ stands for the interior unit normal vector to $\partial\Omega$. Notice we used the hypothesis on the support of $a$ in order to have the equality $E_{\Omega^c}(\uu(t_{2})) = \mathcal{E}_{\Omega^c}(\uu,t_{2})$.

Let's denote by $h$ the Dirichlet data on $\partial\Omega$ given by $h = h_+ + h_-$, with $h_\pm = \Lambda^\pm_a\ff$, and recall Remark \ref{remark0}. We can use the construction near the boundary from the previous section and get a representation of the transmitted wave, $\uu_T = \uu^+_T + \uu^-_T$, which satisfies that $\uu_T\cong \uu$, with $\cong$ meaning equality up to a smoothing operator applied to $h$. Notice this parametrix solutions are only constructed in neighborhoods of $(t^\pm_1,x^\pm_1)$. For times outside this neighborhoods, we can approximate $\uu$ using the FIOs of section 4.1. Nevertheless, for such intervals of time we know the solution is smooth, therefore the more energetic part of $\uu$ is contained precisely in $\uu^+_T$ and $\uu^-_T$. Then   
we can estimate the RHS of \eqref{energy2} and get, modulo compact operators applied to $h_\pm$,
\begin{equation}\label{e1}
\begin{aligned}
E_{\Omega^c}(\uu,T) &\cong 2\mathfrak{R}\int_{[0,T]\times\partial \Omega}\frac{\partial u_T}{\partial t}\frac{\partial \overline{u}_T}{\partial\nu}dtdS\\
&\cong 2\mathfrak{R}(P_th,-(N_+ h^+ + N_- h^-))\\
& \cong \sum_{\sigma=\pm}\mathfrak{R}(-2N^*_\sigma P_th_\sigma,h_\sigma),
\end{aligned}
\end{equation}
where $(\cdot,\cdot)$ stands for the inner product in $L^2(\RR\times\RR^{n-1})$. 
Notice there is no cross terms between the functions $h_\pm$ in the right hand side. This is because the wave front set of $N^*_\pm$ is contained in a small neighborhood of $(t_1^\pm,x^\pm_1,1,(\xi_1^\pm)')$, while the wave front set of $h_\mp$ lies close to $(t_1^\mp,x^\mp_1,1,(\xi_1^\mp)')$. Since both bicharacteristics don't intersect each other the above wave front sets are disjoint and consequently $N^*_\pm P_th_\mp$ are smooth functions. Therefore, those terms involve compact operators applied to the functions $h_\pm$.

From the definition of the DN map and the above energy relation we deduce that
$$E_{\Omega^c}(\uu(T)) = \sum_{\sigma=\pm}\mathfrak{R}(M_\sigma h_\sigma,h_\sigma)$$
with $M_\pm$ two $\Psi$DOs of order 2 with the same principal symbol
\begin{align*}
\sigma_p(M_\pm) &= \sigma_p(-N^*_\pm P_t) = -2\cdot\overline{\text{i}\sqrt{c^{-1}\tau^2 - |\xi'|^2}}\cdot (-\text{i}\tau) \\
&= 2\tau\sqrt{c^{-2}\tau^2 - |\xi'|^2}.
\end{align*}
Notice that $h_{\pm}$ are compactly supported and as we mentioned before, their wave front sets are respectively contained in neighborhoods of the points $(t_1^\pm,x^\pm_1,1,(\xi_1^\pm)')$ where the respective bicharacteristic through $(0,x_0,1,\pm\xi_0)$ hits the boundary. Furthermore, their essential supports lie in the hyperbolic region  $c|\xi'|<\tau$ due to the strictly convexity of $\Omega$ which makes the bicharacteristics to cross $\partial\Omega$ non-tangentially. Then, $\sigma(M_\pm)\geq C|(\tau,\xi')|^2$ and we can
apply Garding's inequality to obtain
\begin{equation}\label{energy_ineq1+}
\begin{aligned}
E_{\Omega^c}(\uu(T)) &\geq C_1\sum_{\sigma=\pm}\|h_\sigma\|^2_{H^{1}(\RR\times\partial\Omega)} - C_2\sum_{\sigma=\pm}\|h_\sigma\|^2_{H^{0}(\RR\times\partial\Omega)}\\
&\geq C_1\|h\|^2_{H^{1}(\RR\times\partial\Omega)} - C_2\sum_{\sigma=\pm}\|h_\sigma\|^2_{H^{0}(\RR\times\partial\Omega)}.
\end{aligned}
\end{equation}

Let ${\bf X} =\text{diag}(X,X)$, with $X$ a zero order $\Psi$DO with essential support in a conic neighborhood of $(x_0,\xi_0)\in S^*\bar{\Omega}\setminus 0$, and such that ${\bf X}\ff$ vanishes outside $\Omega$. From the last inequality and by choosing $h = \Lambda_{a}{\bf X}\ff = \Lambda_{a}^+{\bf X}\ff+\Lambda^-_{a}{\bf X}\ff$, the continuity and stability of the measurement operator in \eqref{FIO_reg} and \eqref{Homan_stability} give us that  
\begin{equation}\label{energy_ineq2}
\|{\bf X}\ff\|^2_{\mathcal{H}(\Omega)} \leq CE_{\Omega^c}(\uu(T))  + C\| {\bf X}\ff\|^2_{H^{0}(\Omega)\times H^{-1}(\Omega)}.
\end{equation}


By compactness of $WF(\ff)\cap S^*\bar{\Omega}$, in a conic neighborhood of such set we can consider a finite pseudo-differential partition of unity $1=\sum\chi_j$ of symbols of $\Psi$DO's $X_j$, localizing in conic neighborhoods of a finite number of points $(x_j,\xi^j)\in WF(\ff)\cap S^*\bar{\Omega}$. Then, $\ff = (I-\sum {\bf X}_j)\ff + \sum {\bf X}_j\ff$, where $WF(\ff)\cap WF(I-\sum {\bf X}_j)=\emptyset$, thus from the inequality above we get 
$$\|\ff\|_{\mathcal{H}(\Omega)}^2\leq C\sum_j E_{\Omega^c}(e^{t{\bf P}_a}{\bf X}_j\ff(T)) + C\|\ff\|^2_{H^{0}(\Omega)\times H^{-1}(\Omega)}.$$
We can substitute $e^{t{\bf P}_a}{\bf X}_j\ff$ by $Q_j{\bf X}_j\ff$ in the right hand side, with $Q_j$ denoting the parametrix operator for the wave equation localized near $(x_j,\xi_j)$, since both are equal up to a compact operator applied to $\ff$ (see \cite[\S 4.10]{TATbrain}). By means of Egorov's Theorem (see for instance \cite[Theorem 10.1]{IntroMLA1994}) there exist zero order $\Psi$DO's, $\tilde{{\bf X}}_j$, such that $Q_j{\bf X}_j = \tilde{{\bf X}}_jQ_j$ modulo a smoothing operator, therefore the exact solution $\uu = e^{t\text{\bf P}_a}\ff$ of \eqref{general IBVP} satisfies
\begin{equation}\label{energy_ineq3}
\|\ff\|_{\mathcal{H}(\Omega)}\leq C\|\uu(T)\|_{H^1(\Omega^c)\oplus L^2(\Omega^c)}+ C\|\ff\|_{H^{0}(\Omega)\times H^{-1}(\Omega)}.
\end{equation}
To get rid of the second term in the right hand side we use a classical argument that requires to show 
$$\mathcal{H}(\Omega)\ni \ff\mapsto \uu(T)\in H^1(\RR^n\setminus\Omega)\oplus L^2(\RR^n\setminus\Omega)$$
is an injective bounded map. The continuity follows from \cite[Proposition 1]{Homan} when the domain is a large ball containing $\text{supp}(\uu(T))$, which we know exists by finite speed of propagation. Assume there is $\ff\in \mathcal{H}(\Omega)$ such that 
$$u(T,x) = 0,\;\forall x\in\RR^n\setminus\Omega.$$
By finite domain of dependence of the damped wave equation, $u$ must vanish in $\{(t,x)\in(0,\infty)\times\RR^n: \text{dist}(x,\partial\Omega)>|T-t|\}$, and since $\uu(0,\cdot) = \ff $ with $\text{supp}\ff\subset \bar{\Omega}$, by the same reason $u=0$ in $\{(t,x)\in(0,\infty)\times\RR^n: \text{dist}(x,\partial\Omega)>t\}$. Intersecting both light cones we in fact have that $u$ vanishes in $[0,3T/2]\times\{x\in\RR^n:\text{dist}(x,\partial\Omega)>T/2\}$. By Tataru's unique continuation \cite[Theorem 4]{Tataru1} and since we assume $T>T_0(\Omega)>2T_1(\Omega)$, with 
$$T_1(\Omega) = \sup_{x\in \Omega} d(x,\partial\Omega),$$
and $d(x,\partial\Omega)$ the infimum of the lengths of curves with respect to $c^{-2}dx^2$ starting at $x$ and ending at $\partial\Omega$, we obtain $u = 0$ in a neighborhood of $\{3T/4\}\times \RR^n$. From Proposition 1 in \cite{Homan}, it then follows by solving the initial value problem backward from $t=3T/4$ to $t=0$, that $u=0$ in $[0,3T/4]\times\Omega$, so in particular $\ff\equiv0$. 

On the other hand, the inclusion $\mathcal{H}(\Omega)\hookrightarrow H^{0}(\Omega)\times H^{-1}(\Omega)$ is compact, so by \cite[Proposition V.3.1]{Taylor} there is some $C>1$ which depends on $a$, such that
$$\|\ff\|_{\mathcal{H}(\Omega)} \leq C\|\uu(T)\|_{H^1(\RR^n\setminus\Omega)\oplus L^2(\RR^n\setminus\Omega)}.$$
To obtain the energy $E_{\Omega^c}(\uu(T)) $ in the right hand side of the last estimate and conclude the proof, we apply Poincare's inequality in $B\setminus\Omega$ with $B$ a large ball containing $\Omega$ and such that $\uu(T)$ vanishes in its complement (exists by finite speed of propagation). 
\end{proof}

\begin{figure}
\centerline{\includegraphics[scale=0.5]{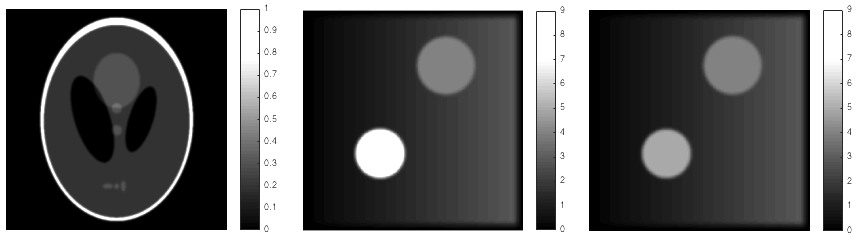}}
\caption{{\small (left) Initial condition to be reconstructed. (center) Attenuation coefficient for the first example. (right) Smaller attenuation coefficient for the second example.}}
\label{fig:settings}
\end{figure}

From the already proven proposition we get
\begin{align*}
\mathcal{E}_{\Omega}(\uu,T) = \mathcal{E}_{\RR^n}(\uu,T) - \mathcal{E}_{\Omega^c}(\uu,T)
&\leq \|\ff\|^2_{\mathcal{H}(\Omega)} - \|\ff\|^2_{\mathcal{H}(\Omega)}/C\\
&\leq (1-1/C)\|f\|^2_{\mathcal{H}(\Omega)},
\end{align*}
where we used that the extended energy is preserved in $\RR^n$, i.e. $\mathcal{E}_{\RR^n}(\uu,T) = \mathcal{E}_{\RR^n}(\uu,0)$.
Finally by recalling \eqref{K_ineq} we conclude that
\begin{align*}
\|{\bf K}_a\ff\|^2_{\mathcal{H}(\Omega)}= E_{\Omega}({\bf w}(0))\leq \mu\|\ff\|^2_{\mathcal{H}(\Omega)},
\end{align*}
with $0<\mu<1$, therefore ${\bf K}_a$ is a contraction in the norm induced by $\|\cdot\|_{\mathcal{H}(\Omega)}$ and there is convergence of the Neumann series.

\begin{figure}[ht]
\centerline{\includegraphics[scale=0.25]{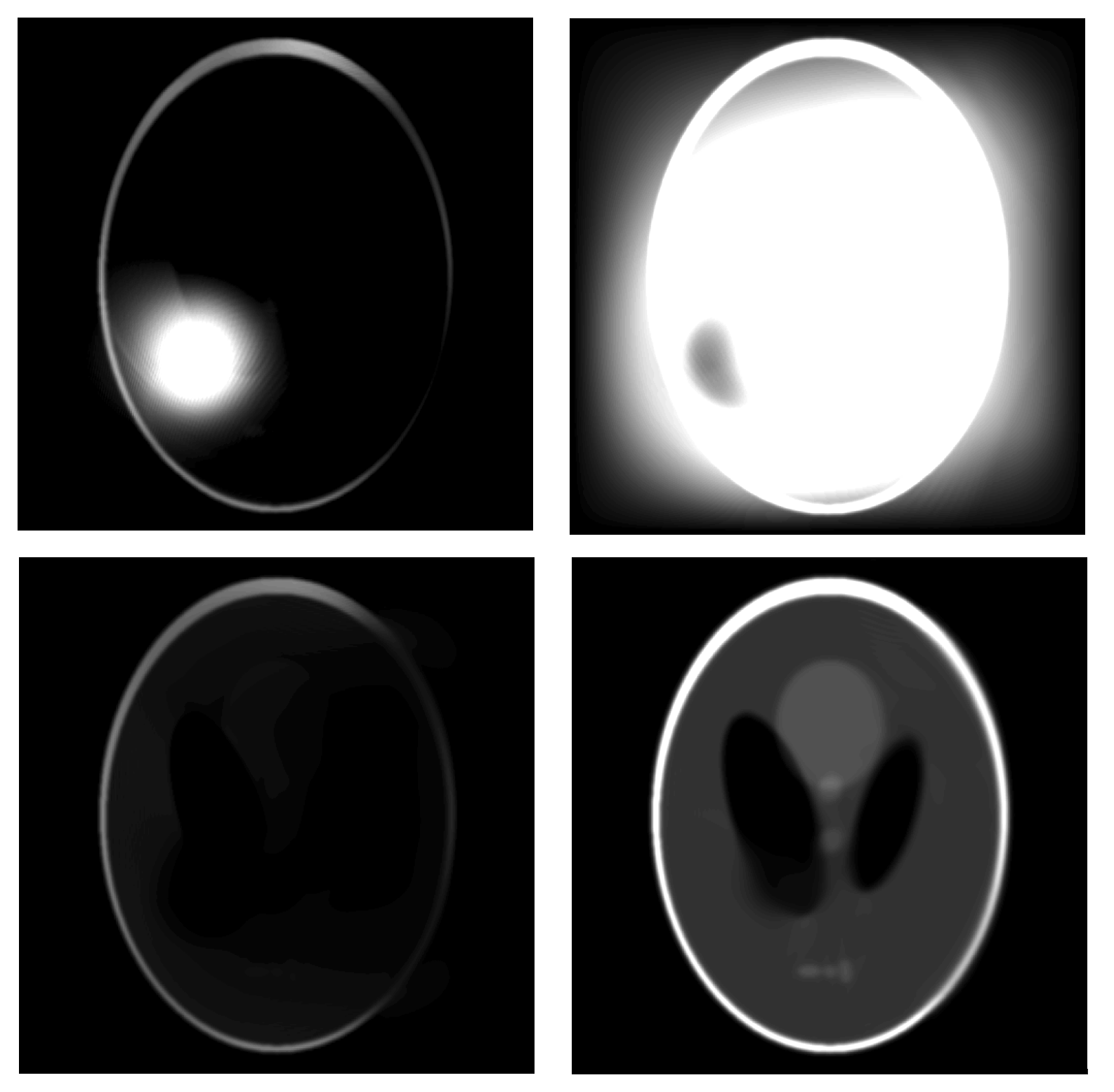}}
\caption{\small
The first row is Homan's reconstruction with 1 and 4 terms. The respective errors are 126\% and 322\%. The second row is the reconstruction following our new method for 1 and 100 terms where we got errors of 76\% and 12\% respectively.}
\label{ex_high_att}
\end{figure}

\begin{remark}
Notice that $\mu$ depends on the attenuation coefficient $a$ and it becomes closer to 1 as $\|a\|_\infty$ goes to infinity. Then, the convergence of the Neumann series gets worse when large attenuations are involved.
\end{remark}
\bigskip

\centerline{\sc 6. Numerics}

We carried out some numerical simulations with the purpose of illustrating the theoretical result obtained in this article. We used a grid of 501x501 to discretize the domain $\Omega = [-1,1]^2$, a variable sound speed given by the function $c(x) = 1 + 0.2\sin(2\pi x_1) + 0.1\cos(2\pi x_2)$, and a regularized version of the Shepp-Logan phantom as initial condition. We also chose a measurement time of $T=3$, and since the convergence rate of this reconstruction method is slower due to the presence of the damping parameter we used 100 terms in the Neumann series reconstruction. The forward propagation of the wave was discretized following the work done in \cite{TATnumeric} which implements PML conditions to simulate the propagation of waves in the free space, while for the back projection of the boundary data we used an standard second order accurate scheme.

\begin{figure}
\centerline{\includegraphics[scale=0.25]{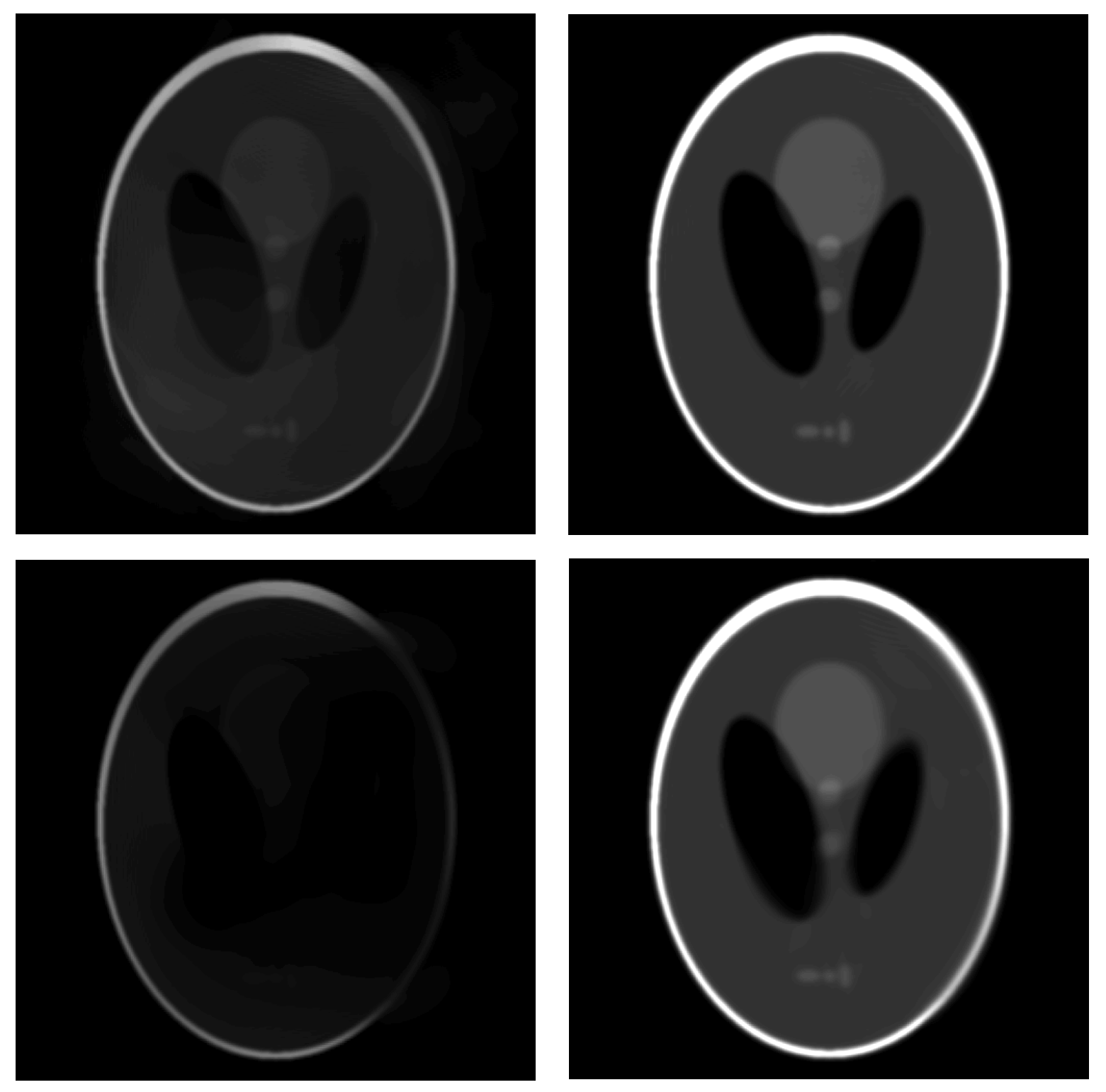}}
\caption{\small
For Homan's reconstruction in the first row, the errors are 46\% and 5\% for 1 and 8 terms, while following ours reconstruction method we get 75\% and 11\% for 1 and 100 terms respectively.}
\label{ex_low_att}
\end{figure}

We consider a damping coefficient formed by a small background attenuation that 
increases toward the x-axis 
and two regularized disks with higher attenuations as you can see in Figure \ref{fig:settings}. Namely, we let $a(x)$ to be of the form
$$a(x) = F\big( d_1\chi_{D_1}(x) + d_2\chi_{D_2}(x) + d_3(1+x_1)\chi_{\Omega\setminus(D_1\cup D_2)}(x)\big)$$
where $\chi$ is the characteristic function, $D_1 = \{x\in\RR^2 : |(x_1,x_2) -(\frac{1}{3},-\frac{1}{2}) |^2 < 0.05)\}$, $D_2 = \{x\in\RR^2 : |(x_1,x_2) - (-\frac{1}{3},\frac{1}{3}) |^2 < 0.07)\}$, and $F$ stands for a regularization function which also brings the attenuation coefficient to zero in a smooth way close to $\partial\Omega$. The idea is to compare the images obtained by the Time Reversal (this is only one term in the Neumann series) and the images when more summands are considered for Homan's reconstruction method, this is solving \eqref{time_reversal} in the back propagation process; and the method developed in this article.

In the first example we take $d_1 = 9$, $d_2 = 4$ and $d_3 = 1.5$. Figure \ref{ex_high_att} shows the results of the Time Reversal (left column) and when adding more terms in the Neumann series (right column). It's evident the improvement in the reconstruction by using our new approach since the first method diverges while ours converges, at a slow rate though and it requires significantly many more iterations in order to achieve a reasonable error.

For the second example (Figure \ref{ex_low_att}) we consider a lower attenuation coefficient with $d_1 = 5$, $d_2 = 4$ and $d_3 = 1.5$. In this case Homan's Neumann series does converge and of course according to the theory it shows a faster convergences than the Neumann series presented here.\\

\centerline{\sc 7. Acknowledgments}

The author would like to thanks Gunther Uhlmann for all the support and advising throughout the writing of this paper and Carlos Montalto for the many fruitful conversations about the subject. Also thanks Francois Monard and Sebastian Acosta for the help and guidance in the numerical part.

\renewcommand{\refname}{\centerline{\large \sc References}}
\bibliographystyle{plain}
{\footnotesize
\bibliography{biblio}}
\end{document}